\documentclass[a4paper,12pt,oneside]{article}

\usepackage{hyperref}
\usepackage[T1]{fontenc}
\usepackage{authblk}
\usepackage{mathtools}
\usepackage{mathrsfs}
\usepackage{amssymb}
\usepackage{xifthen}
\usepackage{tikz}
\usepackage{amsthm}
\usepackage{latexsym}
\usepackage[all]{xy}
\usepackage{amscd,graphicx,amsfonts}
\usepackage{xcolor}
\usepackage{geometry}

\newtheorem{theorem}{Theorem}[section]

\newtheorem{lemma}[theorem]{Lemma}

\newtheorem{remark}[theorem]{Remark}
\newtheorem{defn}[theorem]{Definition}

\newgeometry{vmargin={28mm},hmargin={20mm,17mm}}

\hypersetup{
colorlinks,
linkcolor={orange},
citecolor={magenta},
urlcolor={blue}
}

\title{Unipotence of a double commutator with transvections}

\author{Ihechukwu Chinyere\thanks{Corresponding author:
\texttt{i.chinyere@up.ac.za; ihechukwu@aims.ac.za}}}

\affil{Department of Mathematics and Applied Mathematics, University of Pretoria,
Hatfield 0028, Pretoria, South Africa}

\begin{document}

\maketitle

\begin{abstract}
\noindent
Let $R$ be a nonzero commutative ring with identity, $n\geq3$, and $\sigma\in GL_n(R)$. We prove that there exist nonidentity transvections $\tau_1,\tau_2\in GL_n(R)$ such that $\big[[\sigma,\tau_1],\tau_2\big]=I_n$. Hence, Kourovka Notebook Problem 10.46 has a stronger affirmative answer over commutative rings with identity. The proof uses rank-one operators and the existence of a nonzero functional annihilating a standard basis vector and its image under $\sigma$.
\end{abstract}


\section{Introduction}

Let $R$ be a nonzero commutative ring with identity, let $n\geq3$, and let $V=R^n$. We write $GL_n(R)$ for the group of invertible $n\times n$ matrices over $R$. For $g,h\in GL_n(R)$, we use the commutator convention $[g,h]=ghg^{-1}h^{-1}$.

\medskip

We use the term \emph{transvection} in the rank-one sense, as made precise in Definition~\ref{def:transvection} below; the direction vectors used in the construction are unimodular, so the resulting transvections are of the usual rank-one type (see Remark~\ref{rem:unimodular}).

\medskip

The problem considered here is a special case of Problem 10.46 in the \emph{Kourovka Notebook}, proposed by V.~M. Petechuk, which appears in the 21st issue of the Notebook (2026) \cite{Kourovka2026}. The problem asks whether, for every $\sigma\in GL_n(R)$ with $n\geq3$, there exist transvections $\tau_1,\tau_2$ such that $\big[[\sigma,\tau_1],\tau_2\big]$ is unipotent. We prove the stronger statement that, over every nonzero commutative ring with identity, the double commutator can in fact be chosen to equal $I_n$.

\medskip

The problem is known to fail in general once commutativity is dropped: Muliarchyk \cite{Muliarchyk} showed that over arbitrary associative rings with identity the answer is negative, by constructing a noncommutative algebra and an invertible matrix for which every nonempty iterated commutator with nonidentity unipotents is nonunipotent and of infinite order. Over fields, however, the answer is positive, and his Proposition~5.1 shows the double commutator can always be taken to be the identity, via a rank-one operator argument. The present paper adapts that argument from fields to arbitrary nonzero commutative rings with identity, showing that the identity double commutator persists in this broader setting, including rings with zero divisors.

\medskip

A related but more restrictive line of work is due to Petechuk and Petechuk \cite{Petechuk2020}, who studied such commutators over division rings in terms of residual and fixed submodules. By contrast, the construction given here works over any nonzero commutative ring with identity and produces, for every $\sigma$, a second nonidentity transvection for which the iterated commutator is exactly $I_n$. The key point making this possible for $n\geq3$ is the existence of a nonzero functional annihilating a prescribed standard basis vector together with one further vector.

\medskip

The commutativity assumption enters explicitly through the construction of such a functional: if $j,k$ are distinct indices, then $w_kw_j-w_jw_k=0$ for $w_j,w_k\in R$. Hence the present argument is specifically adapted to commutative coefficient rings.

\section{Main result}

We now set up the notation and rank-one machinery needed to state and prove our main theorem. The argument proceeds in three steps: we first record some elementary identities for rank-one operators (Lemma~\ref{lem:rankone}), then use commutativity of $R$ together with the hypothesis $n\geq3$ to produce a functional vanishing on a prescribed pair of vectors (Lemma~\ref{lem:functional}), and finally combine these ingredients to construct the two transvections $\tau_1,\tau_2$ asserted in Theorem~\ref{main}.

\medskip

Throughout, let $V=R^n$ and $V^*=\operatorname{Hom}_R(V,R)$. We write $e_1,\ldots,e_n$ for the standard basis of $V$, and $e_1^*,\ldots,e_n^*\in V^*$ for the corresponding dual basis, determined by $e_j^*(e_\ell)=\delta_{j\ell}$ for $1\leq j,\ell\leq n$.

\begin{defn}\label{def:rankone}
For $u\in V$ and $f\in V^*$, define the operator $u\otimes f\in\operatorname{End}_R(V)$ by $(u\otimes f)(x)=f(x)u$ for $x\in V$.
\end{defn}

The following elementary identities will be used repeatedly.

\begin{lemma}\label{lem:rankone}
Let $u,v\in V$, $f,g\in V^*$, and $\sigma\in GL_n(R)$. Then:
\begin{enumerate}
\item[\textup{(i)}] $(u\otimes f)^2=f(u)(u\otimes f)$;
\item[\textup{(ii)}] $(u\otimes f)(v\otimes g)=f(v)(u\otimes g)$;
\item[\textup{(iii)}] $\sigma^{-1}(u\otimes f)\sigma=(\sigma^{-1}u)\otimes(f\circ\sigma)$;
\item[\textup{(iv)}] $\sigma(u\otimes f)\sigma^{-1}=(\sigma u)\otimes(f\circ\sigma^{-1})$.
\end{enumerate}
\end{lemma}

\begin{proof}
For (i), $(u\otimes f)^2(x)=f\big(f(x)u\big)u=f(x)f(u)u=f(u)(u\otimes f)(x)$ for all $x\in V$. For (ii), $(u\otimes f)(v\otimes g)(x)=f\big(g(x)v\big)u=g(x)f(v)u=f(v)(u\otimes g)(x)$ for all $x\in V$. For (iv), $\sigma(u\otimes f)\sigma^{-1}(x)=\sigma\big(f(\sigma^{-1}x)u\big)=f(\sigma^{-1}x)(\sigma u)=(f\circ\sigma^{-1})(x)(\sigma u)=\big((\sigma u)\otimes(f\circ\sigma^{-1})\big)(x)$ for all $x\in V$; the proof of (iii) is identical with $\sigma$ replaced by $\sigma^{-1}$.
\end{proof}

In the proof of Theorem~\ref{main} below, parts (i) and (iv) of Lemma~\ref{lem:rankone} are used directly, while part (ii) is used in the computation of the products $NM$ and $MN$. Part (iii) is not needed in the sequel and is recorded only for completeness.

\begin{defn}\label{def:transvection}
An element $\tau\in GL_n(R)$ is called a transvection if $\tau=I_n+u\otimes f$ for some $u\in V$ and $f\in V^*$ satisfying $f(u)=0$ and $u\otimes f\neq0$.
\end{defn}
\begin{remark}\label{rem:unimodular}
If $u\in V$ is unimodular (i.e., there exists $\phi\in V^*$ with $\phi(u)=1$), then $u\otimes f\neq0$ whenever $f\neq0$: choosing $x\in V$ with $f(x)\neq0$ and applying $\phi$ to $(u\otimes f)(x)=f(x)u$ gives $\phi\big((u\otimes f)(x)\big)=f(x)\neq0$. Hence, when $MN\neq0$, the transvection $\tau_2=I_n+(\sigma v)\otimes g$ constructed in the proof has the usual rank-one form, since $\sigma v$ is unimodular.
\end{remark}

\begin{lemma}\label{lem:standardbasis}
If $0\neq f\in V^*$, then $e_i\otimes f\neq0$ for every $1\leq i\leq n$.
\end{lemma}

\begin{proof}
Choose $x\in V$ such that $f(x)\neq0$. Then $(e_i\otimes f)(x)=f(x)e_i\neq0$.
\end{proof}

\begin{defn}\label{def:unipotent}
An element $g\in GL_n(R)$ is called \emph{unipotent} if $g-I_n$ is nilpotent.
\end{defn}

The next lemma is the point at which both commutativity of $R$ and the assumption $n\geq3$ enter the construction.

\begin{lemma}\label{lem:functional}
Let $R$ be a nonzero commutative ring with identity and let $n\geq3$. Fix a standard basis vector $v=e_i$ and let $w\in R^n$. Then there exists a nonzero $f\in V^*$ such that $f(v)=f(w)=0$.
\end{lemma}

\begin{proof}
Choose distinct indices $j,k$ different from $i$. Write $w=\sum_{\ell=1}^n w_\ell e_\ell$. If at least one of $w_j,w_k$ is nonzero, define $f=w_k e_j^*-w_j e_k^*$. Then $f(v)=0$ and, by commutativity of $R$, $f(w)=w_kw_j-w_jw_k=0$. Moreover $f(e_j)=w_k$ and $f(e_k)=-w_j$, and since at least one of $w_j,w_k$ is nonzero, $f\neq0$. If $w_j=w_k=0$, simply take $f=e_j^*$. Again $f\neq0$ and $f(v)=f(w)=0$.
\end{proof}

\begin{remark}
The condition $n\geq3$ in Lemma~\ref{lem:functional} is needed for this particular construction of $f$, which requires two indices $j,k$ distinct from $i$. For $n=2$ the construction is vacuous: e.g. with $R=\mathbb Z$, $v=e_1$, and $w=e_2$, any $f=ae_1^*+be_2^*$ satisfying $f(v)=f(w)=0$ must have $a=b=0$, so no nonzero functional annihilates both vectors. This shows only that the present method of proof breaks down at $n=2$; it does not by itself imply that the conclusion of Theorem~\ref{main} fails for $n=2$.
\end{remark}

\begin{theorem}\label{main}
Let $R$ be a nonzero commutative ring with identity, let $n\geq3$, and let $\sigma\in GL_n(R)$. Then there exist nonidentity transvections $\tau_1,\tau_2\in GL_n(R)$ such that $\big[[\sigma,\tau_1],\tau_2\big]=I_n$.
\end{theorem}

\begin{proof}
Choose a standard basis vector $v=e_i$ and put $w=\sigma v$. By Lemma~\ref{lem:functional}, there exists a nonzero $f\in V^*$ such that $f(v)=f(\sigma v)=0$. Set $N=v\otimes f$ and $\tau_1=I_n+N$. Since $f(v)=0$, Lemma~\ref{lem:rankone}(i) gives $N^2=0$, and hence $\tau_1^{-1}=I_n-N$. By Lemma~\ref{lem:standardbasis}, $N\neq0$, so $\tau_1$ is a nonidentity transvection.

\medskip

Now put $M=\sigma N\sigma^{-1}$. By Lemma~\ref{lem:rankone}(iv), $M=(\sigma v)\otimes(f\circ\sigma^{-1})$. Since $(f\circ\sigma^{-1})(\sigma v)=f(v)=0$, Lemma~\ref{lem:rankone}(i) gives $M^2=0$. Moreover, by Lemma~\ref{lem:rankone}(ii),
$$
NM=(v\otimes f)\big((\sigma v)\otimes(f\circ\sigma^{-1})\big)=f(\sigma v)\,\big(v\otimes(f\circ\sigma^{-1})\big)=0,
$$
since $f(\sigma v)=0$. Also by Lemma~\ref{lem:rankone}(ii),
$$
MN=\big((\sigma v)\otimes(f\circ\sigma^{-1})\big)(v\otimes f)=(f\circ\sigma^{-1})(v)\,\big((\sigma v)\otimes f\big)=f(\sigma^{-1}v)\,\big((\sigma v)\otimes f\big).
$$
Since $NM=0$, we have $(MN)^2=M(NM)N=0$.

\medskip

Let $A=[\sigma,\tau_1]$. Since $\tau_1^{-1}=I_n-N$, we obtain

$$
A=(I_n+M)(I_n-N)=I_n+M-N-MN.
$$

We claim that $A$ commutes with $I_n+MN$. Indeed, using $M^2=N^2=NM=(MN)^2=0$,
$$
(A-I_n)MN=(M-N-MN)MN=M^2N-NMN-(MN)^2=0,
$$
$$
MN(A-I_n)=MN(M-N-MN)=MNM-MN^2-(MN)^2=0.
$$
Consequently, $A(I_n+MN)=(I_n+MN)A$.

\medskip

Suppose first that $MN\neq0$. From the computation above, $MN=(\sigma v)\otimes g$, where $g=f(\sigma^{-1}v)f$, and $g(\sigma v)=f(\sigma^{-1}v)f(\sigma v)=0$. Set $\tau_2=I_n+MN=I_n+(\sigma v)\otimes g$. Since $MN\neq0$, we have $g\neq0$, and $g(\sigma v)=0$, so $\tau_2$ is a nonidentity transvection. Moreover, $\sigma v$ is unimodular, since $v=e_i$ is unimodular and $\sigma$ is invertible: if $\phi\in V^*$ satisfies $\phi(v)=1$, then $(\phi\circ\sigma^{-1})(\sigma v)=\phi(v)=1$. Hence, by Remark~\ref{rem:unimodular}, $\tau_2$ is of the usual rank-one type. Since $A$ commutes with $\tau_2$, we obtain $\big[[\sigma,\tau_1],\tau_2\big]=[A,\tau_2]=I_n$.

\medskip

It remains to consider the case $MN=0$. Then $A=I_n+M-N$. Since $M^2=N^2=MN=NM=0$, we have $AN=N=NA$. Hence $A$ commutes with $\tau_1=I_n+N$. Taking $\tau_2=\tau_1$ (so that, as noted in Remark~\ref{rem:unimodular}, the direction vector of $\tau_2$ is $v$ rather than $\sigma v$), we again obtain $\big[[\sigma,\tau_1],\tau_2\big]=[A,\tau_2]=I_n$.
\end{proof}

\textbf{Use of LLMs.}
ChatGPT and Claude assisted with editing and presentation. The mathematical arguments and their verification remain the responsibility of the author.

\end{document}